\documentclass[reqno, 11pt]{amsart}
\usepackage{}
\usepackage{bbm}
\usepackage{url}
\usepackage{stmaryrd}
\usepackage{mathrsfs}
\usepackage{cases}
\usepackage{amsfonts}
\usepackage{amssymb}
\usepackage{amsmath}
\usepackage{enumerate}
\usepackage{tikz}
\usepackage{extarrows}
\usepackage{mathtools}
\usepackage{citeref}
\allowdisplaybreaks[4]
\newtheorem{theorem}{Theorem}[section]
\newtheorem{lemma}[theorem]{Lemma}
\newtheorem{proposition}[theorem]{Proposition}
\newtheorem{corollary}[theorem]{Corollary}
\theoremstyle{definition}
\newtheorem{definition}[theorem]{Definition}

\newtheorem{remark}[theorem]{Remark}
\numberwithin{equation}{section}

\let\a=\alpha \let\al=\alpha
\let\b=\beta

\let\g=\gamma
\let\G=\Gamma  \let\Ga=\Gamma
\let\fy=\infty \let\i=\infty 

\let\l=\ell

\let\s=\sigma

\let\la=\lambda
\let\La=\Lambda
\let\om=\omega

\let\th=\theta

\let\wt=\widetilde

\def\bbR{\mathbb{R}}

\def\bbN{\mathbb{N}}
\def\bbZ{\mathbb{Z}}

\def\scrP{\mathscr{P}}

\def\scrF{\mathscr{F}}
\def\scrS{\mathscr{S}}

\def\calD{\mathcal {D}}
\def\calF{\mathcal {F}}

\def\calI{\mathcal {I}}
\def\calL{\mathcal {L}}

\def\calT{\mathcal {T}}
\def\calS{\mathcal {S}}

\def\rr{{\mathbb R}}
\def\rd{{{\bbR}^d}} 
\def\rdd{{{\rr}^{2d}}}  

\def\zd{{{\mathbb{Z}}^d}}
\def\zdd{{{\mathbb{Z}}^{2d}}}

\def\mpqs{M^{p,q}_s}

\def\wpqs{W^{p,q}_s}

\newcommand{\be}{\begin{equation*}}
	\newcommand{\ee}{\end{equation*}}
\newcommand{\ben}{\begin{equation}}
	\newcommand{\een}{\end{equation}}
\newcommand{\bn}{\begin{enumerate}}
	\newcommand{\en}{\end{enumerate}}
\newcommand{\bs}{\backslash}
\newcommand{\lan}{\langle}
\newcommand{\ran}{\rangle}

\def\wpq{W^{p,q}}        
\def\wpqm{W^{p,q}_m}
\def\wpoqom{W^{p_1,q_1}_m}

\def\wptqt{W^{p_2,q_2}}
\def\mpqm{{M^{p,q}_m}}

\begin{document}
\title[Schr\"{o}dinger operator on Wiener Amalgam spaces in the full range]
{Schr\"{o}dinger operator on Wiener Amalgam spaces in the full range}
\address{School of Mathematics and Statistics, Xiamen University of Technology, Xiamen, 361024, P.R.China} 
\email{guopingzhaomath@gmail.com}
\author{GUOPING ZHAO}
\address{School of Science, Jimei University, Xiamen, 361021, P.R.China}
\email{weichaoguomath@gmail.com}
\author{WEICHAO GUO}

\begin{abstract}
Using the technique of Gabor analysis, an equivalent characterization is established for the boundedness $e^{i\Delta}: \wpoqom\rightarrow \wptqt$, 
where $0<p_i,q_i\leq\i$ and $m$ is a $v$-moderate weight. 
The sharp exponents for the boundedness is also characterized in the case of power weight. 
\end{abstract}

\maketitle


\section{Introduction}
Consider the Cauchy problem of the following Schr\"{o}dinger equation
\begin{equation*}
	\begin{cases*}
		i\partial_tu+\Delta u=F(u),\\
		u(0,x)=u_0.
	\end{cases*}
\end{equation*}
The formal solution is given by
\be
u(t)=e^{it\Delta}u_0-i\int_0^te^{i(t-s)\Delta}F(u(s))ds.
\ee
If we want to establish the local existence with the initial data $u_0$ belonging to certain function space $X$,
a crucial step is to establish the corresponding norm estimate for the free Schr\"{o}dinger operator  on $X$, that is, to establish the estimate of following type:
\ben\label{Int-1}
\|e^{i\Delta}f\|_X\lesssim \|f\|_X.
\een
In the classical research of Schr\"{o}dinger equation, the initial data $u_0$ is usually assumed to belong to a $L^2$-based function spaces, such as
the $L^2$-based Sobolev space $(I-\Delta)^{\frac{-s}{2}}L^2$ or Besov space $B^{2,q}_s$. An important reason is that only $p=2$ allows the estimate $\eqref{Int-1}$ to hold on
$X=(I-\Delta)^{\frac{-s}{2}}L^p$ or $X=B^{p,q}_s$.

The situation changes if we replace the classical dyadic decomposition in the definition of Besov spaces by uniform decomposition.
More precisely,  if we use modulation space $X=\mpqs$ in the estimate \eqref{Int-1}, 
the following boundedness is valid (\cite{WangZhaoGuo2006JoFA, BenyiGroechenigOkoudjouRogers2007JoFA}):
\ben\label{Int-2}
\|e^{i\Delta}f\|_{\mpqs}\leq  C_{p,q}\|f\|_{\mpqs},\ \ \ p,q\in (0,\fy].
\een
Due to this advantage of modulation space, many researchers have begun to use modulation space to study partial differential equations.
We refer the reader to the pioneer works \cite{BenyiOkoudjou2009BotLMS, WangHudzik2007JDE} and to  \cite{BhimaniHaque2023JoFA, Schippa2022JoFA} for some recent progress.

Modulation spaces were introduced firstly by Feichtinger \cite{Feichtinger1983TRUoV} in 1983. It has been regarded as a basic and important class of function spaces
in the field of time-frequency, see \cite{Groechenig2001}.
Comparing with the classical Besov space $B^{p,q}_s$, modulation space $\mpqs$ (see \cite{Triebel1983ZfAuiA} for the equivalent norm) can be regarded as the Besov type space associated with uniform decomposition on the frequency domain.
Another important function space that is closely related to the modulation space (see Lemma \ref{lm-eq-W-M}) is the Wiener amalgam space $\wpqs$ (see Subsection 2.1). 
This space can be regarded as the Triebel-type space associated with uniform decomposition, see \cite{Triebel1983ZfAuiA} for the equivalent norm.

Due to the boundedness in \eqref{Int-2} and the fact that both modulation and Wiener amalgam space are defined based on the uniform decomposition,
an interesting question is to establish the corresponding boundedness on $\wpqs$ of Schr\"{o}dinger operator.
For this direction, one can see Kato-Tomita\cite{KatoTomita2018MN} for a sharp estimate as follows.

\textbf{Theorem A} (\cite{KatoTomita2018MN}) 
Suppose $1\leq p,q\leq\i$. 
Then $e^{i\Delta}: \wpq_{1\otimes v_s}\rightarrow\wpq$ is bounded if and only if $s\geq d|1/p-1/q|$ with the strict inequality when $p\neq q$. 

Here, we write the power weights $v_s(\xi):=(1+|\xi|)^{\frac{s}{2}}$.
From this theorem, one can find that the Schr\"{o}dinger operator is unbounded on $\wpqs$ for $p\neq q$. This is quite different from the case on modulation spaces.
We also refer to \cite{GuoZhao2020JoFA} for the sharp estimate on $\wpqs$ of a class of unimodular Fourier multipliers.
In fact, because of the weaker separation property of Triebel-type spaces, the behaviors of $e^{i\Delta}$ on $\wpqs$  is more difficult to study than that on modulation space.
For instance, the boundedness $e^{i\Delta}: M^{p_1, q_1}_{s}\rightarrow M^{p_2, q_2}$ of full range $p_i, q_i\in (0,\fy]$ has been established in \cite{ZhaoChenFanGuo2016NLAT}.
However, the corresponding boundedness result on Wiener amalgam space of full range is still unknown.
The main goal of this paper is to fill this gap. 
To achieve this goal, our strategy is to first establish an equivalent characterization for the case of general weight. 
To avoid the fact that $\calS(\rd)$ is not dense in some endpoint spaces, such as the case $p=\fy$ or $q=\fy$,
we only consider the action of $e^{i\Delta}$ on Schwartz function spaces.
For the sake of simplicity, we use the statement ``$e^{i\Delta}: W^{p_1,q_1}_{m}\rightarrow W^{p_2,q_2}$'' or $e^{i\Delta}\in \calL(W^{p_1,q_1}_{m},W^{p_2,q_2})$ to express the meaning that
\be
\|e^{i\Delta}f\|_{W^{p_2,q_2}}\leq C\|f\|_{W^{p_1,q_1}_m}\ \ \text{for all}\ f\in \calS(\rd).
\ee
Here, we use $m$ to denote the weight function belonging to the class $\scrP(\rdd)$, see Subsection 2.1 for the precise definition.
We write the dilation operator $\calD_{\la_1,\la_2}m(x,\xi)=m(\la_1 x,\la_2 \xi)$.
For a sequence $\vec{a}=\{a_{k,n}\}$, denote by $(\calT\vec{a})_{k,n}=a_{k-n,n}$ the coordinate transformation of $\vec{a}$.
Our first theorem gives an equivalent characterization of boundedness.
\begin{theorem}[Equivalent characterization]\label{thm-bd-eq}
	Let $p_i, q_i\in (0,\infty]$ for $i=1,2$. Assume that $m\in \scrP(\rdd)$. We have the following equivalent relation
	\be
	e^{i\Delta}\in \calL(W^{p_1,q_1}_{m}(\rd),W^{p_2,q_2}(\rd))
	\Longleftrightarrow 
	\calT\in \calL(l^{(p_1,q_1)}_{\calD_{2, 1/2}m}(\zdd), l^{(p_2,q_2)}(\zdd)).
	\ee
\end{theorem}

Thanks to the above equivalent relation, one can turn to studying the corresponding discrete inequality without having to consider the Schr\"{o}dinger operator directly.
Our second theorem is a sharp exponents characterization for the case of $m=1\otimes v_s$. 
It is known that the boundedness $e^{i\Delta}\in \calL(W^{p_1,q_1}_{1\otimes v_s}(\rd),W^{p_2,q_2}(\rd))$ is the most interesting case, especially in the crossing field of harmonic analysis and PDEs.

\begin{theorem}[Sharp exponents characterization]\label{thm-bd-se}
	Let $s\in \bbR$, $p_i, q_i\in (0,\infty]$ for $i=1,2$.  
	Denote $A=d(1/p_2-1/q_1)$, $B=d(1/q_2-1/p_1)$.
	The boundedness 
	\be
	e^{i\Delta}\in \calL(W^{p_1,q_1}_{1\otimes v_s}(\rd),W^{p_2,q_2}(\rd))
	\ee
	holds if and only if $1/p_2\leq 1/p_1$ and
	\be
	s\geq A\vee B\vee (A+B)\vee 0
	\ee
    with the strict inequality if one of the following cases happens:
    \bn
    \item $A>0\geq B$;
    \item $B>0\geq A$;
    \item $A, B>0$, $p_1=p_2$.
    \en  
\end{theorem}
Using this theorem, one can conclude directly the following boundedness of Schr\"{o}dinger operator on Wiener amalgam spaces without potential loss.
\begin{corollary}
  $e^{i\Delta} \in \calL(W^{p_1,q_1}(\rd),W^{p_2,q_2}(\rd))$ holds if and only if
  \begin{equation*}
  q_1\leq p_2, \text{ and } p_1\leq q_2 \wedge p_2.
  \end{equation*}
\end{corollary}
\begin{remark}
Some remarks about our main theorems are listed as follows:
\bn
\item 
The proof of Theorem \ref{thm-bd-eq} is based on the Gabor analysis of Wiener amalgam space and the magic formula in Lemma \ref{lm-STFT-Schrodinger} associated with Schr\"{o}dinger operator and the time-frequency shift.
We point out that this formula was also used in \cite{KatoKobayashiIto2012TMJSS} for giving some conservation quantity associated with Schr\"{o}dinger operator. 
\item 
Theorem \ref{thm-bd-se} is an essential extension of Theorem A in \cite{KatoTomita2018MN}.
In our theorem of full range, more endpoint cases create new difficulties. 
The equivalent characterization in Theorem \ref{thm-bd-eq} allows us to see the nature more clearly and solve it.
\item 
All the conclusions in Theorems \ref{thm-bd-eq} and \ref{thm-bd-se} can be extended to the fixed time Schr\"{o}dinger operator $e^{it_0\Delta}$ for any fixed $t_0\in\bbR$. 
\item 
The method of this article is also applicable in the study of the following boundedness:
\[
  e^{i\Delta} \in \calL(X_1, X_2), 
  \text{ with }X_j=W^{p_j,q_j}_{m_j} \text{ or } M^{p_j,q_j}_{m_j}, j=1,2.
\]
\item 
Although the magic formula in Lemma \ref{lm-STFT-Schrodinger} seems to hold only for very special unimodular Fourier multipliers, 
the similar equivalent characterizations as in Theorem \ref{thm-bd-eq} of general unimodular Fourier multipliers are still expected to be valid with appropriate modification.
However, the exponents characterizations as in Theorem \ref{thm-bd-se} should be much more difficult for more general unimodular Fourier multipliers.
\en
\end{remark}

The rest of this paper is organized as follows. 
In section 2, we first introduce some basic definitions and properties of the function spaces used throughout this paper.
We also prepare some useful conclusions of the Gabor analysis on Wiener amalgam space.
Section 3 is devoted to the equivalent characterization of the boundedness $e^{i\Delta}: \wpoqom\rightarrow \wptqt$.
As mentioned above, the key tools are the Gabor analysis and the magic formula of Schr\"{o}dinger operator.
The sharp exponents characterization is proved in Section 4. In some critical case, the boundedness of Schr\"{o}dinger operator
is found to be equivalent to the boundedness of fractional integral operators.

\section{Preliminaries}
	\subsection{Function spaces}
	In order to introduce the function spaces, we first recall some definitions of weight.
	Recall that a weight is a positive and locally integral function on $\rdd$.
	A weight function $m$ is called $v$-moderate if there exists another weight function $v$ such that
	\be
	m(z_1+z_2)\leq C_v^m v(z_1)m(z_2),\ \ \ \ z_1,z_2\in \rdd,
	\ee
	where $v$ belongs to the class of submultiplicative weight, that is, $v$ satisfies
	\be
	v(z_1+z_2)\leq v(z_1)v(z_2),\ \ \ \ z_1,z_2\in \rdd.
	\ee
	Moreover, in this paper,
	we assume that $v$ has at most polynomial growth.
	If the associated weight $v$ is implicit, we call that $m$ is moderate,
	and use the notation $\scrP(\rdd)$ to denote the cone of all wights which are moderate	in $\rdd$.
	We also assume that every $m\in \scrP(\rdd)$ is continuous since there exists a continuous weight $m_1$ such that $m\sim m_1$. 
	We refer to \cite{1979Gewichtsfunktionen} for the origin of the $v$-moderate weights. See also \cite{Heil2003} for more properties of these weights.

   \begin{definition}[Continuous mixed-norm spaces.]
   	Let $m\in \scrP(\rdd)$, $p,q\in (0,\fy]$. The weighted mixed-norm space $L^{p,q}_m(\rdd)$
   	consists of all Lebesgue measurable functions on $\rdd$ such that the (quasi-)norm
   	\be
   	\begin{split}
   		\|F\|_{L^{p,q}_m(\rdd)}
   		= &
   		\|Fm\|_{L^{p,q}(\rdd)}
   		=
   		\|F(x,\xi)m(x,\xi)\|_{L^{p,q}_{x,\xi}}
   		\\
   		= &
   		\left(\int_{\rd}\left(\int_{\rd}|F(x,\xi)|^pm(x,\xi)^pdx\right)^{q/p}d\xi\right)^{1/q}
   	\end{split}
   	\ee
   	is finite, with the usual modification when $p=\fy$ or $q=\fy$. We also use the notation
   	\be
   	\begin{split}
   		\|F\|_{L^{(p,q)}_m(\rdd)}
   		= &
   		\|Fm\|_{L^{(p,q)}(\rdd)}
   		=
   		\|F(x,\xi)m(x,\xi)\|_{L^{q,p}_{\xi,x}}
   		\\
   		= &
   		\left(\int_{\rd}\left(\int_{\rd}|F(x,\xi)|^qm(x,\xi)^qd\xi\right)^{p/q}dx\right)^{1/p}.
   	\end{split}
   	\ee
   	Then the mixed-norm space $L^{(p,q)}_m(\rdd)$ can be defined by the similar way as above.
   \end{definition}
	
	In order to introduce the definitions of modulation and Wiener amalgam spaces, we recall some definitions and notations in time-frequency analysis. 
For $x,\xi\in\bbR^d$, the translation operator $T_x$ and the modulation operator $M_{\xi}$ are given by
\[
T_xf(t)=f(t-x), \text{ and } M_{\xi}f(t)=e^{2\pi i \xi\cdot t}f(t).
\]
For $z:=(x, \xi)$,
we also use the notation $\pi(z):=M_{\xi}T_x$, which is also known as the time-frequency shift on the phase plane.

\begin{definition}
	Let $g\in\calS(\bbR^d)\setminus\{0\}$, the short-time Fourier transform(STFT) of $f\in \calS'(\bbR^d)$ with respect to the window $g$ is defined by
	\[
	V_gf(x,\xi)
	=\int_{\bbR^d}f(t)\overline{g(t-x)}e^{-2\pi i\xi\cdot t}d t
	=\langle f, M_{\xi}T_xg \rangle
	=\langle f, \pi(z)g \rangle,
	\]
	where the integral makes sense for nice function $f$.
\end{definition}

The so-called fundamental identity of time-frequency analysis is as follows:
\ben\label{fid}
V_gf(x,\xi)=e^{-2\pi ix\cdot \xi}V_{\hat{g}}\hat{f}(\xi,-x),\ \ \ (x,\xi)\in \rdd.
\een

It can be observed that the STFT of a distribution $f$ takes both the decay and smooth properties into account.
The modulation space can be regarded as a collection of functions sharing the same decay and smooth properties.
Now, we recall the definition of modulation space.
\begin{definition}[Modulation space]\label{df-M}
	Let $0<p,q\leq \infty$, $m\in \mathscr{P}(\rdd)$.
	Given a non-zero window function $\phi\in \calS(\rd)$, the (weighted) modulation space $M^{p,q}_m(\rd)$ consists
	of all $f\in \calS'(\rd)$ such that the (quasi-)norm
	\be
	\begin{split}
		\|f\|_{M^{p,q}_m(\rd)}&:=\|V_{\phi}f\|_{L^{p,q}_m(\rdd)}
		=\left(\int_{\rd}\left(\int_{\rd}|V_{\phi}f(x,\xi)m(x,\xi)|^{p} dx\right)^{{q}/{p}}d\xi\right)^{{1}/{q}}
	\end{split}
	\ee
	is finite.
	If $m\equiv 1$, we write $M^{p,q}$ for short. We also write $M^{p,q}_s$ for the case $m=1\otimes v_s$.
\end{definition}
Recall that the above definition of $M^{p,q}_m$ is independent of the choice of window function $\phi$.
We refer the readers to \cite{Groechenig2001} for the case $(p,q)\in\lbrack 1,\infty ]^{2}$,
and to \cite[Theorem 3.1]{GalperinSamarah2004ACHA} for full range $(p,q)\in (0,\infty ]^{2}$.
More precisely, a class of admissible windows denoted by $\mathfrak{M}^{p,q}_v$ was found in \cite{GalperinSamarah2004ACHA}, such that
every window function $g\in \mathfrak{M}^{p,q}_v$ yields the equivalent quasi-norm on $\mpqm$.

\begin{definition}[The space of admissible windows]\label{D.window-class}
	Let $0<p, q\leq\infty$, $r=\min\{1,p\}$ and $s=\min\{1, p, q\}$.
	Let $m$ be $v$-moderate.
	For $r_1, s_1>0$, denote
	\[w_{r_1, s_1}(x,w)=v(x,\omega)(1+|x|)^{r_1}(1+|\omega|)^{s_1}.\]
	Then $\mathfrak{M}_v^{p,q}$ the admissible windows for the modulation space $M_m^{p,q}$  can be defined as:
	\[ \mathfrak{M}_v^{p,q}:=\bigcup_{\substack{r_1>d/r\\s_1>d/s\\1\leq p_1<\infty}} M^{p_1}_{w_{r_1,s_1}}.\]
\end{definition}

Next, we recall the definition of the $\wpqm$ space.

\begin{definition}[Wiener amalgam space]\label{df-W}
	Let $0<p,q\leq \infty$, $m\in \mathscr{P}(\rdd)$.
	Given a non-zero window function $\phi\in \calS(\rd)$, the (weighted) Wiener amalgam space $W^{p,q}_m(\bbR^d)$ consists of all tempered distributions $f\in\calS'(\bbR^d)$ such that 
	\be
	\begin{split}
		\|f\|_{W^{p,q}_m(\bbR^d)}&:=\|V_{\phi}f\|_{L^{(p,q)}_m(\rdd)}
		=\left(\int_{\bbR^d}\left(\int_{\bbR^d}|V_{\phi}f(x,\xi)m(x,\xi)|^{q} d\xi\right)^{{p}/{q}}dx\right)^{{1}/{p}}
	\end{split}
	\ee
	is finite.
	We write briefly $W^{p,q}$ for the case $m(x,\xi)\equiv 1$.
	We also use the notation $W^{p,q}_s$ for the case $m=1\otimes v_s$.
\end{definition}

Recall that there is a close connection between modulation and Wiener amalgam spaces.

\begin{lemma}\label{lm-eq-W-M}
	Denote $\wt{m}(x,\xi):=m(\xi,-x)$. We have $\wpqm=\calF M^{q,p}_{\wt{m}}$. More precisely, the following relation is valid:
	\be
	\|f\|_{\wpqm}\sim \|V_{\phi}f(x,\xi)m(x,\xi)\|_{L^{q,p}_{\xi,x}}=\|V_{\check{\phi}}\check{f}(x,\xi)m(\xi,-x)\|_{L^{q,p}_{x,\xi}}\sim \|\check{f}\|_{M^{q,p}_{\wt{m}}}.
	\ee
\end{lemma}
Due to the above relation, many conclusions can be automatically converted from modulation spaces to Wiener amalgam spaces.
For instance, the definition of $\wpqm$ is independent of the window $g\in \calF\mathfrak{M}_{\wt{v}}^{q,p}$, where  $\wt{v}(x,\xi):=v(\xi,-x)$.

\subsection{Gabor analysis on modulation and Wiener amalgam spaces}
In this subsection, we recall an important time-frequency tool on modulation and Wiener amalgam spaces.
In fact, the functions or distributions in modulation and Wiener amalgam spaces can be characterized by the summability and decay properties of their Gabor coefficients.
We first list some important operators with their basic properties.

\begin{definition}
	Assume that $g,\g\in L^2(\bbR^d)$, $\al,\b>0$ and $\Ga=\a\bbZ^d\times\b\bbZ^d$. The coefficient operator or analysis operator $C^{\a,\b}_g$
	is defined by
	\be
	C^{\a,\b}_gf=\{\langle f, M_{\b n}T_{\a k}g \rangle\}_{k,n\in\bbZ^d}.
	\ee
	The synthesis operator or reconstruction operator $D^{\a,\b}_{\g}$ is defined by
	\be
	D^{\a,\b}_{\g}\vec{c}=\sum_{k\in \bbZ^d}\sum_{n\in \bbZ^d} c_{k,n}M_{\b n}T_{\a k}\g.
	\ee
	The Gabor frame operator $S^{\a,\b}_{g,\g}$ is defined by
	\be
	S^{\a,\b}_{g,\g}f
	 =D^{\a,\b}_{\g}C^{\a,\b}_gf
	 =\sum_{k\in \bbZ^d}\sum_{n\in \bbZ^d} \langle f,  M_{\b n}T_{\a k}g \rangle M_{\b n}T_{\a k}\g.
	\ee
	We also use $C_g$, $D_{\g}$ and $S_{g,\g}$ for short, if the parameters $\al$ and $\b$ are implicit.
\end{definition}

We remark that the definitions of $C_g$ and $D_{\g}$ can be extended if the window functions $g$ and $\g$ are taken suitably.
Next, we give the definitions of discrete mixed-norm spaces.

\begin{definition}[Discrete mixed-norm spaces]
	Let $0<p,q\leq \fy$, $m\in \scrP(\rdd)$.
	The space $l^{p,q}_m(\zdd)$ consists of all sequences $\vec{a}=\{a_{k,n}\}_{k,n\in \zd}$ for which the (quasi-)norm
	\be
	\|\vec{a}\|_{l^{p,q}_m(\zdd)}=\left(\sum_{n\in \zd}\left(\sum_{k\in \zd}|a_{k,n}|^pm(k,n)^p\right)^{q/p}\right)^{1/q}
	\ee
	is finite, with the usual modification when $p=\i$ or $q=\i$. We also use the notation
	\be
	\|\vec{a}\|_{l^{(p,q)}_m(\zdd)}=\left(\sum_{k\in \zd}\left(\sum_{n\in \zd}|a_{k,n}|^qm(k,n)^q\right)^{p/q}\right)^{1/p}.
	\ee
	Then the mixed-norm space $l^{(p,q)}_m(\zdd)$ can be defined by the similar way. For a continuous function $F$ on $\rdd$, we define
	\be
	\|F\|_{\l^{p,q}(\a\bbZ^d\times\b\bbZ^d)}=\|(F(\al k,\b n))_{k,n}\|_{l^{p,q}(\zdd)},\ \ \ 
	\|F\|_{\l^{(p,q)}(\a\bbZ^d\times\b\bbZ^d)}=\|(F(\al k,\b n))_{k,n}\|_{l^{(p,q)}(\zdd)}.
	\ee
\end{definition}
%
Based on the admissible window class mentioned above, we recall the boundedness of $C_g$ and $D_{g}$,
which works on the full range $p,q\in (0,\fy]$, see \cite{GalperinSamarah2004ACHA} for more details.

\begin{lemma}\label{lm-bdM-CgDg}
	Assume that $m$ is $v$-moderate, $p,q\in (0,\fy]$, and $g$ belongs to the subclass $M^{p_1}_{\om_{r_1,s_1}}$ of $\mathfrak{M}^{p,q}_v$.
  For all lattice constants $\a, \b>0$, we have
   \[
      \|C_gf(k,n)m(\a k, \b n)\|_{l^{p,q}}
       =\|V_gf(\a k, \b n)m(\a k, \b n)\|_{l^{p,q}}
      \lesssim \|V_gf\|_{L^{p,q}_m}
      \sim \|f\|_{\mpqm},
   \] 
   and
   \[
     \|D_g\vec{c}\|_{\mpqm}\lesssim \| c_{k,n}m(\a k, \b n) \|_{l^{p,q}}
   \]
   independently of $p,q,m$.
\end{lemma}

Now, we recall the Gabor characterization of modulation space in \cite{GalperinSamarah2004ACHA}.

\begin{lemma}[Gabor characterization of $\mpqm$ with $0<p,q\leq\i$]\label{lm-gc-M}
	\cite[Theorem 3.7]{GalperinSamarah2004ACHA}
	Assume that $m$ is $v$-moderate on $\bbR^{2d}$, $p,q\in (0,\fy]$, $g,\g\in \mathfrak{M}^{p,q}_v$, and that the Gabor frame operator
	$S_{g,\g}=D_{\g}C_g=I$ on $L^2(\bbR^d)$. Then
	\begin{equation*}
	  f
	  =\sum_{k\in \bbZ^d}\sum_{n\in \bbZ^d} \langle f, M_{\b n}T_{\a k}g \rangle  M_{\b n}T_{\a k}\g
	  =\sum_{k\in \bbZ^d}\sum_{n\in \bbZ^d} \langle f, M_{\b n}T_{\a k}\g \rangle  M_{\b n}T_{\a k}g
	\end{equation*}
	with unconditional convergence in $M^{p,q}_m$ if $p,q<\fy$, and with weak-star convergence in $M^{\fy}_{1/v}$ otherwise.
	Furthermore there are constants $A,B>0$ such that for all $f\in M^{p,q}_m$
	\be
	A\|f\|_{M^{p,q}_m}
	\leq
	\left\|  V_gf\cdot m \right\|_{\l^{p,q}(\a\bbZ^d\times\b\bbZ^d)}
	\leq
	B\|f\|_{M^{p,q}_m}.
	\ee
\end{lemma}

Next, we list the corresponding results on Wiener amalgam spaces, which can be verified by using Lemmas \ref{lm-bdM-CgDg} and \ref{lm-gc-M}
and the fundamental identity \eqref{fid}
directly.
We omit the proof here.

\begin{lemma}\label{lm-bdW-CgDg} 
	Assume that $m$ is $v$-moderate, $p,q\in (0,\fy]$, and $g$ belongs to the subclass $\scrF M^{p_1}_{\om_{r_1,s_1}}$ of $\scrF \mathfrak{M}^{q,p}_{\wt{v}}$, where  $\wt{v}(x,\xi):=v(\xi,-x)$.
	For all lattice constants $\a, \b>0$, we have
	\[
	\|C_gf(k,n)m(\a k, \b n)\|_{l^{q,p}_{n,k}} 
	=\|V_gf(\a k, \b n)m(\a k, \b n)\|_{l^{q,p}_{n,k}}
	\lesssim \|V_gf(x,w)m(x,w)\|_{L^{q,p}_{w,x}}
	\sim \|f\|_{\wpqm}.
	\] 
	and
		\[
	\|D_gc\|_{\wpqm}\lesssim \| c(k,n)m(\a k, \b n) \|_{l^{q,p}_{n,k}}
	\]
	independently of $p,q,m$.
\end{lemma}

\begin{lemma}[Gabor frames for Wiener spaces $\wpqm$ with $0<p,q\leq\i$]\label{lm-Gabor-Wpqm}
	Assume that $m$ is $v$-moderate on $\bbR^{2d}$, $p,q\in (0,\fy]$, $g, \g\in \scrF\mathfrak{M}^{q,p}_{\wt{v}}$ with $\wt{v}(x,\xi):=v(\xi,-x)$, and that the Gabor frame operator
	$S_{g,\g}=D_{\g}C_g=I$ on $L^2(\bbR^d)$. Then
	\begin{equation*}
		f
		=\sum_{k\in \bbZ^d}\sum_{n\in \bbZ^d} \langle f, M_{\b n}T_{\a k}g \rangle  M_{\b n}T_{\a k}\g
		=\sum_{k\in \bbZ^d}\sum_{n\in \bbZ^d} \langle f, M_{\b n}T_{\a k}\g \rangle  M_{\b n}T_{\a k}g
	\end{equation*}
	with unconditional convergence in $\wpqm$ if $p,q<\fy$, and with weak-star convergence in $W^{\fy}_{1/v}$ otherwise.
	Furthermore there are constants $A,B>0$ such that for all $f\in W^{p,q}_m$
	\be
	A\|f\|_{\wpqm}
	\leq
	\left\|  V_gf\cdot m \right\|_{\l^{(p,q)}(\a\bbZ^d\times\b\bbZ^d)}
	\leq
	B\|f\|_{\wpqm}.
	\ee
\end{lemma}

The following well known theorem provides a way to find the Gabor frame of $L^2(\bbR^d)$.
Recall that $\|g\|_{W(L^{\fy},L^{1})(\bbR^d)}=\sum_{n\in \bbZ^d}\|g\chi_{Q+n}\|_{L^{\fy}}$ with $Q=[0,1]^d$.
\begin{lemma}(Walnut \cite{Walnut1992JMAA})\label{lm-frame-L2}
	Suppose that $g\in W(L^{\fy},L^{1})(\bbR^d)$ satisfying
	\be
	A\leq \sum_{k\in \bbZ^d}|g(x-\al k)|^2\leq B\ \ \ \ a.e.
	\ee
	for constants $A,B\in (0,\fy)$.
	Then there exists a constant $\b_0$ depending on $\al$ such that
	$\mathcal {G}(g,\al,\b):= \{T_{\al k}M_{\b n}g\}_{k,n\in \bbZ^d}$ is a Gabor frame of $L^2(\bbR^d)$ for all $\b\leq \b_0$.
\end{lemma}

In order to find the dual window in a suitable function space, we recall the following conclusion.
\begin{lemma}(\cite{Groechenig2001})\label{lm-frame-invertible}
	Assume that $g\in M^1_v(\rd)$ and that $\{T_{\al k}M_{\b n}g\}_{k,n\in \zd}$ is a Gabor frame for $L^2(\rd)$.
	Then the Gabor frame operator $S_{g,g}^{\al,\b}$ is invertible on $M^1_v(\rd)$. As a consequence, $S_{g,g}^{\al,\b}$
	is invertible on all modulation spaces $M^{p,q}_m(\rd)$ for $1\leq p,q\leq \fy$ and $m\in \scrP(\rdd)$.
\end{lemma}

    Based on the above lemmas, we give a lemma used in our proof of Theorem \ref{thm-bd-eq}.
    \begin{lemma}\label{lm-eqnorm-wpqm}
    	Suppose that $p_i,q_i\in (0,\fy]$, $m_i\in \scrP(\rd)$, $i=1,2$. For any $g_1, g_2\in \calS(\rd)\bs \{0\}$, there exists a constant $N\in \bbN$ such that for all $f\in \calS(\rd)$
    	\be
    	\|f\|_{W^{p_i,q_i}_{m_i}} \sim \|V_{g_i}f(\frac{k}{N},\frac{n}{N})m_i(\frac{k}{N},\frac{n}{N})\|_{l^{(p_i,q_i)}(\zdd)}.
    	\ee
    \end{lemma}

    \begin{proof}
    For nonzero Schwartz functions $g_1$ and $g_2$, there exists a constant $\a=1/\wt{N}$ with $\wt{N}\in\bbN$, such that
    \be
    0<A_0\leq \sum_{k\in \bbZ^d}|g_i(x-\frac{k}{\wt{N}})|^2\leq B_0<\fy,\ \ \  x\in \rd.
    \ee
    Using Lemma \ref{lm-frame-L2}, one can find a constant $\b=\al/L$ with $L\in \bbN$ such that
    $\mathcal {G}(g_i,\al,\b):= \{T_{\al k}M_{\b n}g_i\}_{k,n\in \bbZ^d}$ is a Gabor frame of $L^2(\bbR^d)$ for $i=1,2$, respectively.
    Denote by $\g_1$ and $\g_2$ be the canonical dual widow functions of $g_1$ and $g_2$, respectively.
    Using Lemma \ref{lm-frame-invertible}, we find that $\g_i\in M^1_{v_s\otimes v_s}$ for any sufficiently large $s>0$.
    It follows that $\g_i\in \scrF M^1_{v_s\otimes v_s}\subset \scrF\mathfrak{M}^{q_i,p_i}_{\wt{v_i}}$, where $\wt{v_i}(x,\xi):=v_i(\xi,-x)$, $i=1,2$.
      We also have $S_{g_i,\g_i}=D_{\g_i}C_{g_i}=I$ on $L^2(\bbR^d)$.

    Using Lemma \ref{lm-Gabor-Wpqm}, one can find two constants $A,B>0$ such that for every $f\in \calS(\rd)$ 
      	\be
      	A\|f\|_{W^{p_i,q_i}_{m_i}}
      	\leq
      	\left\|  V_{g_i}f\cdot m_i \right\|_{\l^{(p_i,q_i)}(\a\bbZ^d\times\b\bbZ^d)}
      	\leq
      	B\|f\|_{W^{p_i,q_i}_{m_i}},\ \ \ \  i=1,2.
      	\ee
      	Denote $N=L\wt{N}$.
        Note that $\a\zd\times \b\zd \subset \b\zd\times\b\zd=\zd/N\times \zd/N$. We obtain
        \be
        A\|f\|_{W^{p_i,q_i}_{m_i}}
        \leq
        \left\|  V_{g_i}f\cdot m_i \right\|_{\l^{(p_i,q_i)}(\a\bbZ^d\times\b\bbZ^d)}\leq \|V_{g_i}f(\frac{k}{N},\frac{n}{N})m_i(\frac{k}{N},\frac{n}{N})\|_{l^{(p_i,q_i)}(\zdd)}.
        \ee
        On the other hand, the sampling property (see Lemma \ref{lm-bdW-CgDg}) yields that
        \be
        \|V_{g_i}f(\frac{k}{N},\frac{n}{N})m_i(\frac{k}{N},\frac{n}{N})\|_{l^{(p_i,q_i)}(\zdd)}\leq B_1 \|f\|_{W^{p_i,q_i}_{m_i}}.
        \ee
      	The desired conclusion follows by the above two estimates.
    \end{proof}

%

\section{Equivalent characterization of $e^{i\Delta}$ on Wiener amalgam spaces}

In this section, we give the proof of Theorem \ref{thm-bd-eq}, showing that the boundedness $e^{i\Delta}\in \calL(W^{p_1,q_1}_m,W^{p_2,q_2})$
is equivalent to the corresponding boundedness on discrete mixed spaces of coordinate transformation $\calT$.
To achieve this goal, we first give a calculation of STFT associated with Sch\"{o}rdinger operator. See also \cite{KatoKobayashiIto2012TMJSS}.

\begin{lemma}[STFT of $e^{-i\pi|D|^2}$]\label{lm-STFT-Schrodinger}
	For $f\in \scrS'$, $\phi\in\scrS$, we have 
	\[
	V_{e^{-i\pi|D|^2}\phi}(e^{-i\pi|D|^2}f)(x,\om)
	=
	e^{-i\pi|\om|^2} \cdot V_{\phi}f(x-\om,\om).
	\]
\end{lemma}
\begin{proof}
	Note that $e^{-i\pi|D|^2}\phi\in\scrS$.
	Using the foundation identity of time-frequency analysis, we find that
	\begin{align*}
		V_{e^{-i\pi|D|^2}\phi}(e^{-i\pi|D|^2}f)(x,\om)
		&=e^{-2\pi i x\cdot \om} V_{\widehat{e^{-i\pi|D|^2}\phi}}(\widehat{e^{-i\pi|D|^2}f})(\om, -x)
		\\
		 &=e^{-2\pi i x\cdot \om} \langle \widehat{e^{-i\pi|D|^2}f},  M_{-x}T_{\om}\widehat{e^{-i\pi|D|^2}\phi} \rangle
		\\	  
		&=e^{-2\pi i x\cdot \om} \langle e^{-i\pi|\xi|^2}\widehat{f}(\xi),  M_{-x}T_{\om}e^{-i\pi|\xi|^2}\widehat{\phi}(\xi) \rangle
		\\
		&=e^{-2\pi i x\cdot \om+i\pi|\om|^2} \langle \widehat{f},  M_{\om-x}T_{\om}\widehat{\phi} \rangle
		\\
		&=e^{-2\pi i x\cdot \om+i\pi|\om|^2} V_{\widehat{\phi}}\widehat{f}(\om,\om-x)
		 =e^{-i\pi|\om|^2} \cdot V_{\phi}f(x-\om,\om).
	\end{align*}
\end{proof}

For the convenient of the proof of Theorem \ref{thm-bd-eq}, we would like to use the boundedness 
of $e^{-i\pi|D|^2}$ rather than that of the standard Schr\"{o}dinger operator $e^{i\Delta}$. 
The following lemma give the relationship between the boundedness of these two operators.

\begin{lemma}\label{lm-bd-eq}(Boundedness of $e^{-i\pi|D|^2}$ and $e^{i\Delta}$)
	Let $p_i, q_i\in (0,\infty]$ for $i=1,2$. Assume that $m\in \scrP(\rdd)$. We have the following equivalent relation
   \be
     e^{i\Delta}\in \calL(W^{p_1,q_1}_{m}(\rd),W^{p_2,q_2}(\rd))
     \Longleftrightarrow 
     e^{-i\pi|D|^2}\in \calL(W^{p_1,q_1}_{\calD_{2,1/2}m}(\rd),W^{p_2,q_2}(\rd)).
     \ee
\end{lemma}
\begin{proof}
	By using the scaling property of Fourier transform and a direct calculation, we have 
	\begin{align*}
	  e^{i\Delta}f
	  &= \scrF^{-1}(e^{-4\pi i|\om|^2}\hat{f}(\om))
	    =\scrF^{-1}(e^{-i\pi|2\om|^2}\hat{f}(\om))
	    =2^d\scrF^{-1}(e^{-i\pi|2\om|^2}\widehat{\calD_{2}f}(2\om))
	  \\
	  &= \calD_{1/2}\scrF^{-1}(e^{-i\pi|\om|^2}\widehat{\calD_{2}f}(\om))
	    =\calD_{1/2}e^{-i\pi|D|^2}\calD_{2}f.
	\end{align*}
    Note that 
	\begin{align}\label{pf.scaling-STFT}
	V_{\calD_{1/2}\phi}\calD_{1/2} f(x,\om)
	=2^{d}V_{\phi}f(x/2,2\om).
\end{align}
    For any $f\in \scrS$ we have
    \be
    \begin{split}
    	\|e^{i\Delta}f\|_{W^{p_2,q_2}}
    	&= 
    	\|\calD_{1/2}e^{-i\pi|D|^2}\calD_{2}f\|_{W^{p_2,q_2}}
    	=
    	\|V_{\calD_{1/2}\phi}(\calD_{1/2}e^{-i\pi|D|^2}\calD_{2}f)(x,\om)\|_{L^{(p_2,q_2)}}
    	\\
    	&= 
    	2^d\|V_{\phi}(e^{-i\pi|D|^2}\calD_{2}f)(x/2,2\om)\|_{L^{(p_2,q_2)}} 
    	\sim 
    	\|V_{\phi}(e^{-i\pi|D|^2}\calD_{2}f)(x,\om)\|_{L^{(p_2,q_2)}} 
    	\\
    	&=
    	\|e^{-i\pi|D|^2}\calD_{2}f\|_{W^{p_2,q_2}}.
    \end{split}
    \ee
    On the other hand, we find
    \be
    \begin{split}
    	\|f\|_{W^{p_1,q_1}_m}
    	&= 
    	\|\calD_{1/2}\calD_2f\|_{W^{p_1,q_1}_m}
    	= 
    	\|V_{\calD_{1/2}\phi}(\calD_{1/2}\calD_2f)(x,\om)m(x,\om)\|_{L^{(p_1,q_1)}}
    	\\
    	&=
    	2^d\|V_{\phi}(\calD_2f)(x/2,2\om)m(x,\om)\|_{L^{(p_1,q_1)}}
        \sim 
        \|V_{\phi}(\calD_2f)(x,\om)m(2x,\om/2)\|_{L^{(p_1,q_1)}}
        \\
        &=\|\calD_2f\|_{W^{p_1,q_1}_{\calD_{2,1/2}m}}.
    \end{split}
    \ee
    From the above two estimates, for any $f\in \scrS$, we have
    \be
    \begin{split}
    	e^{i\Delta}\in \calL(W^{p_1,q_1}_{m},W^{p_2,q_2})
    	&\Longleftrightarrow
    	\|e^{i\Delta}f\|_{W^{p_2,q_2}}\lesssim \|f\|_{W^{p_1,q_1}_m}\ \ \text{for any}\ f\in \scrS
    	\\
    	&\Longleftrightarrow
    	\|e^{-i\pi|D|^2}\calD_{2}f\|_{W^{p_2,q_2}}\lesssim \|\calD_2f\|_{W^{p_1,q_1}_{\calD_{2,1/2}m}}\ \ \text{for any}\ f\in \scrS
    	\\
    	&\Longleftrightarrow e^{-i\pi|D|^2}\in \calL(W^{p_1,q_1}_{\calD_{2,1/2}m},W^{p_2,q_2}).
    \end{split}
    \ee
\end{proof}
Now, we turn to the characterization of the boundedness of $e^{-i\pi|D|^2}$.

\begin{proposition}[Equivalent characterization for $e^{-i\pi|D|^2}$]\label{pp-bd-eq}
  Let $p_i, q_i\in (0,\infty]$ for $i=1,2$. Assume that $m\in \scrP(\rdd)$. The following equivalent relation is valid:
  \be
  e^{-i\pi|D|^2}\in \calL(W^{p_1,q_1}_m(\rd),W^{p_2,q_2}(\rd))  \Longleftrightarrow \calT\in \calL(l^{(p_1,q_1)}_m(\zdd), l^{(p_2,q_2)}(\zdd)).
  \ee
\end{proposition}
\begin{proof}
	We first verify the ``$\Longrightarrow$'' direction. For this purpose, we choose a non-negative function $h\in \calS(\rd)\bs \{0\}$ satisfying $\text{supp}\hat{h}\subset B(0,1/8)$ and $\|h\|_{L^2}=1$.
	For any truncated (only finite nonzero terms) non-negative sequence $\vec{a}=\{a_{k,n}\}_{(k,n)\in \zdd}$, we define
	\be
	f=D_{h}\vec{a}=\sum_{(k,n)\in \zdd}a_{k,n}\pi(k,n)h.
	\ee
	By a direct calculation, we find that
	\be
	\begin{split}
		V_{h}f(k,n)
		= &
		\sum_{j,l}a_{j,l}\lan \pi(j,l)h, \pi(k,n)h\ran 
		= 
		\sum_{j}a_{j,n}\lan \pi(j,n)h, \pi(k,n)h\ran 
		\\
     	= &
		\sum_{j}a_{j,n}\lan T_jh, T_kh\ran \geq a_{k,n}\lan T_kh, T_kh\ran=a_{k,n}\|h\|_{L^2}=a_{k,n}.
	\end{split}
	\ee
	Combing this with Lemma \ref{lm-bdW-CgDg}  the sampling property and Lemma \ref{lm-STFT-Schrodinger}, we obtain
    \be
    \begin{split}
    	\|e^{-i\pi|D|^2}f\|_{W^{p_2,q_2}}
    	    	\gtrsim &
    	\|V_{e^{-i\pi|D|^2 h}}(e^{-i\pi|D|^2}f)(k,n)\|_{l^{(p_2,q_2)}}
    	\\
    	= &
    	\|e^{-i\pi|n|^2} \cdot V_hf(k-n,n)\|_{l^{(p_2,q_2)}}
    	\\
    	\geq& 
    	\|a_{k-n,n}\|_{l^{(p_2,q_2)}}=\|\calT \vec{a}\||_{l^{(p_2,q_2)}(\zdd)}.
    	\end{split}	
    \ee
On the other hand, by the boundedness of reconstruction operator (see Lemma \ref{lm-bdW-CgDg}) we obtain
       \be
       \|f\|_{W^{p_1,q_1}_m}=\|D_{h}\vec{a}\|_{W^{p_1,q_1}_m}\lesssim \|\vec{a}\|_{l^{(p_1,q_1)}_m(\zdd)}.
       \ee	
	The ``$\Longrightarrow$'' direction follows by the above two estimates.
	
	Next, we turn to the ``$\Longleftarrow$'' direction. Using Lemma \ref{lm-eqnorm-wpqm},  we can find a constant $N\in \bbN$ such that for any $f\in \calS(\rd)$
	\be
	\|e^{-i\pi|D|^2}f\|_{W^{p_2,q_2}}\sim \|V_{e^{-i\pi|D|^2}h}(e^{-i\pi|D|^2}f)(\frac{k}{N},\frac{n}{N})\|_{l^{(p_2,q_2)}(\zdd)}
	\ee
	and
	\be
	\|f\|_{W^{p_1,q_1}_m}\sim \|V_hf(\frac{k}{N},\frac{n}{N})m(\frac{k}{N},\frac{n}{N})\|_{l^{(p_1,q_1)}(\zdd)}.
	\ee
	From this, in order to obtain the desired conclusion, we only need to verify the following inequality
	\be
	\|V_{e^{-i\pi|D|^2}h}(e^{-i\pi|D|^2}f)(\frac{k}{N},\frac{n}{N})\|_{l^{(p_2,q_2)}(\zdd)} \lesssim \|V_hf(\frac{k}{N},\frac{n}{N})m(\frac{k}{N},\frac{n}{N})\|_{l^{(p_1,q_1)}(\zdd)}.
	\ee
	Using Lemma \ref{lm-STFT-Schrodinger}, this is equivalent to
	\be
	\|V_{h}f(\frac{k-n}{N},\frac{n}{N})\|_{l^{(p_2,q_2)}(\zdd)} \lesssim \|V_hf(\frac{k}{N},\frac{n}{N})m(\frac{k}{N},\frac{n}{N})\|_{l^{(p_1,q_1)}(\zdd)}.
	\ee
	For this constant $N\in \bbN$, we construct a decomposition of $\zd\times \zd$ by
	\be
	\bbZ^d\times \bbZ^d=\bigcup\limits_{(j,l)\in\La} \Gamma_{j,l},
	\ee
	where
	$
	\Gamma_{j,l}:=
	\{(Nk+j,Nn+l), n,k\in\bbZ^d\}$ and $\Lambda:=\{(j,l)\in \bbN^d\times \bbN^d: 0\leq \|j\|_{l^{\i}},\|l\|_{l^{\i}}<N\}.
	$ 
	Write
	\be
	\begin{split}
		\|V_{h}f(\frac{k-n}{N},\frac{n}{N})\|_{l^{(p_2,q_2)}(\zdd)} 
		\leq &
		\sum_{(j,l)\in \La}\|V_{h}f(\frac{k-n}{N},\frac{n}{N})\|_{l^{(p_2,q_2)}(\G_{j,l})}
		\\
		= & 
		\sum_{(j,l)\in \La}\|V_{h}f(k-n+\frac{j-l}{N},n+\frac{l}{N})\|_{l^{(p_2,q_2)}(\zdd)}.
	\end{split}
	\ee
	Using the assumption $\calT\in \calL(l^{(p_1,q_1)}_m(\zdd), l^{(p_2,q_2)}(\zdd))$, the last term above can be dominated by
	\be
	\begin{split}
		&\sum_{(j,l)\in \La}\|V_hf(k+\frac{j-l}{N},n+\frac{l}{N})m(k,n)\|_{l^{(p_1,q_1)}(\zdd)}
		\\
		\lesssim &
		\sum_{(j,l)\in \La}\|V_hf(k+\frac{j-l}{N},n+\frac{l}{N})m(k+\frac{j-l}{N},n+\frac{l}{N})\|_{l^{(p_1,q_1)}(\zdd)}
				\\
		\lesssim &
		\sum_{(j,l)\in \La}\|V_hf(\frac{Nk+j}{N},\frac{Nn+l}{N})m(\frac{Nk+j}{N},\frac{Nn+l}{N})\|_{l^{(p_1,q_1)}(\zdd)}
		\\
		\lesssim &
		\|V_hf(\frac{k}{N},\frac{n}{N})m(\frac{k}{N},\frac{n}{N})\|_{l^{(p_1,q_1)}(\zdd)},
	\end{split}
	\ee
	where we use the fact that
	\be
	m(k,n)\sim m(k+\frac{j-l}{N},n+\frac{l}{N})\ \ \text{for all}\ \ (j,l)\in \La.
	\ee
	We have now completed this proof.
\end{proof}
\textbf{Proof of Theorem \ref{thm-bd-eq}.}
The desired conclusion follows directly by Lemma \ref{lm-bd-eq} and Proposition \ref{pp-bd-eq}.
\section{Sharp exponents characterization of $e^{i\Delta}$ on Wiener amalgam spaces}
Thanks to the equivalent characterization in Theorem \ref{thm-bd-eq}, we first consider the sharp exponents characterization for 
$\calT\in \calL(l^{(p_1,q_1)}_{1\otimes v_s}(\zdd), l^{(p_2,q_2)}(\zdd))$.
The following lemmas play an important role in our proof.
\begin{lemma}[see Lemma 4.4 in \cite{GuoChenFanZhao2019MMJ}]\label{lm-eb-lp}
	$l^{q,s}\subset l^p$ if and only if $s\geq d(1/p-1/q)\vee0$ with strict inequality if $1/p>1/q$. 
\end{lemma}

\begin{lemma}\label{lm-bd-pq-qp}
	Let $0<q\leq p\leq \fy$, $s\in \rr$. We have $\calT\in \calL(l^{(q,p)}_{1\otimes v_s}(\zdd), l^{(p,q)}_{1\otimes v_s}(\zdd))$.
\end{lemma}
\begin{proof}
	Write
	\be
	\begin{split}
		\|\calT\vec{a}\|_{l^{(p,q)}_{1\otimes v_s}}
		= &
		\big(\sum_{k}\big(\sum_{n}|a_{k-n,n}|^{q}\lan n\ran^{sq}\big)^{\frac{p}{q}}\big)^{\frac{1}{p}}
		\\
		= &
		\big(\sum_{k}\big(\sum_{n}|a_{n,k-n}|^{q}\lan k-n\ran^{sq}\big)^{\frac{p}{q}}\big)^{\frac{1}{p}}.
	\end{split}
	\ee
	Applying the Minkowski inequality, the last term can be dominated from above by
	\be
	\begin{split}
		\big(\sum_{n}\big(\sum_{k}|a_{n,k-n}|^{p}\lan k-n\ran^{sp}\big)^{\frac{q}{p}}\big)^{\frac{1}{q}}
		=
		\big(\sum_{n}\big(\sum_{k}|a_{n,k}|^{p}\lan k\ran^{sp}\big)^{\frac{q}{p}}\big)^{\frac{1}{q}}
		=
		\|\vec{a}\|_{l^{(q,p)}_{1\otimes v_s}}.
	\end{split}
	\ee
\end{proof}

For the convenience of the readers, we introduce some notations used in the proof of Proposition \ref{pp-se}.
Denote 
\be
Y=\{(p_1,p_2,q_1,q_2,s)\in (0,\fy]^4\times \rr: \calT\in \calL(l^{(p_1,q_1)}_{1\otimes v_s}(\zdd), l^{(p_2,q_2)}(\zdd))\},
\ee
and let $X$ be the set as follows:
\be
X=\{(p_1,p_2,q_1,q_2,s)\in (0,\fy]^4\times \rr: 1/p_2\leq 1/p_1,\ s\geq A\vee B\vee (A+B)\vee 0\}.
\ee
Define the sets $X_j$, $j=0,1,2,3$, as
\begin{align*}
X_0 &=\{(p_1,p_2,q_1,q_2,s)\in X: A, B\leq 0\}, \\
X_1 &=\{(p_1,p_2,q_1,q_2,s)\in X: A>0, B\leq 0\},\\
X_2 &=\{(p_1,p_2,q_1,q_2,s)\in X: A\leq 0, B>0\},\\
X_3 &=\{(p_1,p_2,q_1,q_2,s)\in X: A, B>0\}.
\end{align*}
Note that $X$ is now the union of the mutually disjoint sets $X_j$, $j=0,1,2,3$.
Denote $Y_j=Y\cap X_j$, $j=0,1,2,3$.
In the proof of Proposition \ref{pp-se}, our strategy is to prove $Y\subset X$ firstly, and then to characterize the set $Y_j$, $j=0,1,2,3$.
Therefore the desired set $Y$ can be characterized by
\be
Y=Y\cap X=Y\cap (\bigcup_{j=0}^3X_j)=\bigcup_{j=0}^3Y_j.
\ee

\begin{proposition}[Sharp exponents for $\calT$]\label{pp-se}
		Let $s\in \bbR$, $p_i, q_i\in (0,\infty]$ for $i=1,2$.  The boundedness 
	\ben\label{pp-se-cd1}
	\calT\in \calL(l^{(p_1,q_1)}_{1\otimes v_s}(\zdd), l^{(p_2,q_2)}(\zdd))
	\een
	holds if and only if $1/p_2\leq 1/p_1$ and
	$s$ satisfies the conditions in Theorem \ref{thm-bd-se}.
\end{proposition}
\begin{proof}
	Recall the definition of $Y$ mentioned above.
	This proof is actually to characterize the set $Y$.
	It can be achieved by the following steps.
	
	\textbf{Step 1.} The goal of this part is to verify that $Y\subset X$. In fact, if we can verify the following claim:
	\ben\label{pp-se-0}
	\calT\in \calL(l^{(p_1,q_1)}_{1\otimes v_s}(\zdd), l^{(p_2,q_2)}(\zdd))
	\Longrightarrow 
	\begin{cases*}
		l^{p_1}\subset l^{p_2},\ \ l^{q_1}_{v_s}\subset l^{p_2},\ \ l^{p_1}_{v_s}\subset l^{q_2},\\
		s\geq d(1/p_2-1/q_1+1/q_2-1/p_1),
	\end{cases*}
	\een
	then the desired conclusion follows by Lemma \ref{lm-eb-lp}.
	Now, we turn to the proof of this claim. 
	If \eqref{pp-se-cd1} holds, we have
	\ben\label{pp-se-1}
	\big(\sum_{k}\big(\sum_{n}|a_{k-n,n}|^{q_2}\big)^{\frac{p_2}{q_2}}\big)^{\frac{1}{p_2}}\lesssim  \big(\sum_k\big(\sum_n|a_{k,n}|^{q_1}\lan n\ran^{sq_1}\big)^{\frac{p_1}{q_1}}\big)^{\frac{1}{p_1}}.
	\een
	The embedding relations $l^{p_1}\subset l^{p_2},\ \ l^{q_1}_{v_s}\subset l^{p_2},\ \ l^{p_1}_{v_s}\subset l^{q_2}$ follows by taking
	\be
	a_{k,n}=
	\begin{cases}
		b_k, &n=0,
		\\
		0, &n\neq 0,
	\end{cases}
    \ \ \ \ \text{and}\ \ 
    a_{k,n}=
    \begin{cases}
    	b_n, &k=0,
    	\\
    	0, &k\neq 0,
    \end{cases}
    \ \ \ \ \text{and}\ \ 
    a_{k,n}=
    \begin{cases}
    	b_k, &k+n=0,
    	\\
    	0, &k+n\neq 0,
    \end{cases}
    \ee
    in the inequality \eqref{pp-se-1}, respectively.
    
    Moreover, for large constant $N$, take
    \be
    a_{k,n}=
    \begin{cases}
    	1, &|k|\leq 2N, |n|\leq N, 
    	\\
    	0, &\text{others}.
    \end{cases}
    \ee
    One can verify that
    \be
    \begin{split}
    	N^{d(1/q_2+1/p_2)}
    	\lesssim &
    	\|(a_{k-n,n})_{k,n}\|_{l^{(p_2,q_2)}}
    	\\
    	\lesssim &
    	\|(a_{k,n} \lan n\ran^s)_{k,n}\|_{l^{(p_1,q_1)}}
    	\lesssim 
    	N^{s+d(1/q_1+1/p_1)+\epsilon},
    \end{split}
    \ee
    for any $\epsilon>0$. 
    Letting $N\rightarrow \i$ and then $\epsilon\rightarrow 0^+$, we obtain the desired conclusion $s\geq d(1/p_2-1/q_1+1/q_2-1/p_1)$.
    
    \textbf{Step 2.} The goal of this part is to verify that $Y_0=Z_0$,
    where $Z_0:=\{(p_1,p_2,q_1,q_2,s)\in (0,\fy]^4\times \rr: 1/p_2\leq 1/p_1, A\leq0, B\leq0, s\geq0\}$.
    It is obvious to see that $Z_0=X_0$.
    To achieve this goal, we only need to verify that $Z_0\subset Y$.
    In this case, we have
    \ben\label{pp-se-2}
    1/p_2\leq 1/p_1,\ \  1/p_2\leq 1/q_1,\ \  1/q_2\leq 1/p_1,\ \ A\vee B\vee (A+B)\vee 0=0.
    \een 
    From this and Lemma \ref{lm-eb-lp}, we have $l^{p_1}\subset l^{q_2}$ and $l^{q_1}_{v_s}\subset l^{p_2}$.
    We only need to verify $\calT\in \calL(l^{(p_1,q_1)}(\zdd), l^{(p_2,q_2)}(\zdd))$ under the above assumptions \eqref{pp-se-2}.
    In fact, we have
    \be
    \|\calT\vec{a}\|_{l^{(p_2,q_2)}}
    \lesssim \|\calT\vec{a}\|_{l^{(p_2,p_1)}} 
    \lesssim \|\vec{a}\|_{l^{(p_1,p_2)}} 
    \lesssim \|\vec{a}\|_{l^{(p_1,q_1)}_{1\otimes{v_s}}},
    \ee
    where we use the embedding relation $l^{p_1}\subset l^{q_2}$ and $l^{q_1}_{v_s}\subset l^{p_2}$ in the first and third inequality respectively, 
    and Lemma \ref{lm-bd-pq-qp} in the second inequality.
    
    \textbf{Step 3.} The goal of this part is to verify that $Y_1=Z_1$, 
    where $Z_1:=\{(p_1,p_2,q_1,q_2,s)\in (0,\fy]^4\times \rr: 1/p_2\leq 1/p_1, A>0\geq B, s>A\}$.
    Actually by the definition of $X_1$, $Z_1=\{(p_1,p_2,q_1,q_2,s)\in X_1: s> A\vee B\vee (A+B)\vee 0=d(1/p_2-1/q_1)\}$.
    
    For $Y_1\subset Z_1$, we only need to prove $s>d(1/p_2-1/q_1)$. 
    Using \eqref{pp-se-0}, we conclude that 
    \be
    (p_1,p_2,q_1,q_2,s) \in Y\cap X_1 \Longrightarrow l^{q_1}_{v_s}\subset l^{p_2}, \ \ 1/p_2>1/q_1.
    \ee
    Form this and Lemma \ref{lm-eb-lp}, we have
    \be
    s> d(1/p_2-1/q_1)=A\vee B\vee (A+B)\vee 0.
    \ee
    
    On the other hand, if $(p_1,p_2,q_1,q_2,s) \in Z_1$, 
    by Lemma \ref{lm-eb-lp} we have
    \be
    1/p_2\leq 1/p_1,\ \ l^{p_1}\subset l^{q_2}, \ \  l^{q_1}_{v_s}\subset l^{p_2}. 
    \ee
    From this, we conclude that
    \be
    \|\calT\vec{a}\|_{l^{(p_2,q_2)}}\lesssim \|\calT\vec{a}\|_{l^{(p_2,p_1)}}\lesssim \|\vec{a}\|_{l^{(p_1,p_2)}}\lesssim \|\vec{a}\|_{l^{(p_1,q_1)}_{1\otimes v_s}},
    \ee
    where we use the embedding relation $l^{p_1}\subset l^{q_2}$ and $l^{q_1}_{v_s}\subset l^{p_2}$ in the first and third inequality respectively, 
    and Lemma \ref{lm-bd-pq-qp} in the second inequality.
    
    \textbf{Step 4.} The goal of this part is to verify that 
    $Y_2=Z_2$, 
    where $Z_2:=\{(p_1,p_2,q_1,q_2,s)\in (0,\fy]^4\times \rr: 1/p_2\leq 1/p_1, B>0\geq A, s>B\}$.
    Actually, $Z_2=\{(p_1,p_2,q_1,q_2,s)\in X_2: s> A\vee B\vee (A+B)\vee 0=d(1/q_2-1/p_1)\}$.
    
    For $Y_2\subset Z_2$, we only need to prove $s>=d(1/q_2-1/p_1)$.
    Using \eqref{pp-se-0}, we conclude that 
    \be
    (p_1,p_2,q_1,q_2,s) \in Y\cap X_2 \Longrightarrow l^{p_1}_{v_s}\subset l^{q_2}, \ \ 1/q_2>1/p_1.
    \ee
    Form this and Lemma \ref{lm-eb-lp}, we have
    \be
    s> d(1/q_2-1/p_1)=A\vee B\vee (A+B)\vee 0.
    \ee
    
    On the other hand, if $(p_1,p_2,q_1,q_2,s) \in Z_2$, 
    by Lemma \ref{lm-eb-lp} we have
    \be
    1/p_2\leq 1/p_1,\ \ l^{p_1}_{v_s}\subset l^{q_2}, \ \  l^{q_1}\subset l^{p_2}.
    \ee
    From this, we conclude that
    \be
    \|\calT\vec{a}\|_{l^{(p_2,q_2)}}\lesssim \|\calT\vec{a}\|_{l^{(p_2,p_1)}_{1\otimes v_s}}\lesssim \|\vec{a}\|_{l^{(p_1,p_2)}_{1\otimes v_s}}\lesssim \|\vec{a}\|_{l^{(p_1,q_1)}_{1\otimes v_s}},
    \ee
    where we use the embedding relation $l^{p_1}_{v_s}\subset l^{q_2}$ and $l^{q_1}\subset l^{p_2}$ in the first and third inequality respectively, 
    and Lemma \ref{lm-bd-pq-qp} in the second inequality.
    
    \textbf{Step 5.} The goal of this part is to verify that $Y_3=Z_3$, where 
    $Z_3=Z_{3,1} \cup Z_{3,2}$ 
    with 
    \begin{align*}
    Z_{3,1}:&=\{(p_1,p_2,q_1,q_2,s)\in (0,\fy]^4\times \rr: A>0, B>0, s>A+B, p_1\leq p_2 \},
    \\
    Z_{3,2}:&=\{(p_1,p_2,q_1,q_2,s)\in (0,\fy]^4\times \rr: A>0, B>0, s=A+B, p_1<p_2 \}.
    \end{align*}
    By the definition of $X_3$, we have 
    $Z_3=X_3 \setminus \{p_1=p_2, s=A+B\}$.
    
    For $Y_3\subset Z_3$, we only need to prove that in $X_3$ the case $\{p_1=p_2, s=A+B\}$ is not true when $\calT$ is bounded.
    We prove this by contradiction, assuming that $p_1=p_2=p, s=A+B$.
    For a large constant $N$, take
    \be
    a_{k,n}=
    \begin{cases}
    	b_n, &|k|\leq 2N, |n|\leq N, 
    	\\
    	0, &\text{others}.
    \end{cases}
    \ee
    Using the inequality \eqref{pp-se-1} and the following estimates
    \be
    \begin{split}
    	\big(\sum_{k}\big(\sum_{n}|a_{k-n,n}|^{q_2}\big)^{\frac{p}{q_2}}\big)^{\frac{1}{p}}
    	\gtrsim &
    	\big(\sum_{|k|\leq N}\big(\sum_{|n|\leq N}|a_{k-n,n}|^{q_2}\big)^{\frac{p}{q_2}}\big)^{\frac{1}{p}}
    	\\
    	= &
    	\big(\sum_{|k|\leq N}\big(\sum_{|n|\leq N}|b_n|^{q_2}\big)^{\frac{p}{q_2}}\big)^{\frac{1}{p}}
    	\sim 
    	N^{d/p}\big(\sum_{|n|\leq N}|b_n|^{q_2}\big)^{\frac{1}{q_2}},
    \end{split}
    \ee
    and
    \be
    \begin{split}
    	\big(\sum_k\big(\sum_n|a_{k,n}|^{q_1}\lan n\ran^{sq_1}\big)^{\frac{p}{q_1}}\big)^{\frac{1}{p}}
    	= &
    	\big(\sum_{|k|\leq 2N}\big(\sum_{|n|\leq N}|a_{k,n}|^{q_1}\lan n\ran^{sq_1}\big)^{\frac{p}{q_1}}\big)^{\frac{1}{p}}
    	\\
    	= &
    	\big(\sum_{|k|\leq 2N}\big(\sum_{|n|\leq N}|b_n|^{q_1}\lan n\ran^{sq_1}\big)^{\frac{p}{q_1}}\big)^{\frac{1}{p}}
    	\sim 
    	N^{d/p}\big(\sum_{|n|\leq N}|b_n|^{q_1}\lan n\ran^{sq_1}\big)^{\frac{1}{q_1}},
    \end{split}
    \ee
    we obtain the embedding relation $l^{q_1}_{v_s}\subset l^{q_2}$ by letting $N\rightarrow \fy$.
    Using this embedding relation and Lemma \ref{lm-eb-lp} with the fact $1/q_2>1/q_1$, we find $s>d(1/q_2-1/q_1)=A+B$.
    This is a contradiction.    
    
    We turn to prove that $Z_3\subset Y_3$.
    This proof is divided into two parts: $Z_{3,1}\subset Y_3$ and $Z_{3,2}\subset Y_3$.
    
    First, let us prove that $Z_{3,1}\subset Y_3$.
    In this case, one can write $s=s_1+s_2$ with $s_1>A=d(1/p_2-1/q_1)$ and $s_2>B=d(1/q_2-1/p_1)$.
    From this and $A,B>0$, by Lemma \ref{lm-eb-lp}, we obtain the embedding relations $l^{p_1}_{v_{s_2}}\subset l^{q_2}$ and $l^{q_1}_{v_{s_1}}\subset l^{p_2}$.
    Then the desired conclusion follows by
    \be
    \|\calT\vec{a}\|_{l^{p_2,q_2}}
    \lesssim 
    \|\calT\vec{a}\|_{l^{p_2,p_1}_{1\otimes v_{s_2}}}
    \lesssim 
    \|\vec{a}\|_{l^{p_1,p_2}_{1\otimes v_{s_2}}}
    \lesssim 
    \|\vec{a}\|_{l^{p_1,q_1}_{1\otimes v_{s_1+s_2}}}=\|\vec{a}\|_{l^{p_1,q_1}_{1\otimes v_{s}}},
    \ee
    where we use the embedding relation $l^{p_1}_{v_{s_2}}\subset l^{q_2}$ and $l^{q_1}_{v_{s_1}}\subset l^{p_2}$ in the first and third inequality respectively, 
    and Lemma \ref{lm-bd-pq-qp} in the second inequality.

    Finally, we turn to the proof of $Z_{3,2}\subset Y_3$.
    In this case, we have $q_2<p_1<p_2<q_1$ and $s=d(1/p_2-1/q_1+1/q_2-1/p_1)$.
    We first consider the endpoint case $q_1=\fy$ and $s=d(1/p_2+1/q_2-1/p_1)$, then the final conclusion follows by using the interpolation argument. 
    In this endpoint case, we write \eqref{pp-se-1} as
    \ben\label{pp-se-3}
    \big(\sum_{k}\big(\sum_{n}|a_{k-n,n}|^{q_2}\big)^{\frac{p_2}{q_2}}\big)^{\frac{1}{p_2}}\lesssim  \big(\sum_k\big(\sup_n|a_{k,n}|\lan n\ran^{s}\big)^{p_1}\big)^{\frac{1}{p_1}}.
    \een
    Denote by $b_k=\sup_n|a_{k,n}|\lan n\ran^{s}$. We have 
    \be
    |a_{k-n,n}|=|a_{k-n,n}|\lan n\ran^s\lan n\ran^{-s}\leq b_{k-n}\lan n\ran^{-s}.
    \ee
    Then the inequality \eqref{pp-se-3} is true if the following one is valid
    \be
    \big(\sum_{k}\big(\sum_{n}\big|\frac{b_{k-n}}{\lan n\ran^s}\big|^{q_2}\big)^{\frac{p_2}{q_2}}\big)^{\frac{1}{p_2}}\lesssim  \big(\sum_k|b_k|^{p_1}\big)^{\frac{1}{p_1}}.
    \ee 
    This is equivalent to
    \ben\label{pp-se-4}
    \|\calI_{\la}\vec{b}\|_{l^{r_2}}
    =
    \bigg\|\big(\sum_{n}\frac{b_{k-n}}{\lan n\ran^{d-\la}}\big)_k\bigg\|_{l^{r_2}}
    \lesssim \|\vec{b}\|_{l^{r_1}}
    \een
    where $r_2=p_2/q_2$, $r_1=p_1/q_2$, $\la=d(1/r_1-1/r_2)$, $\calI_{\la}$ denotes the fractional integral operator of discrete form.
    Note that $r_1, r_2\in (1,\fy)$, $\la\in (0,d)$. The inequality \eqref{pp-se-4} follows by the boundedness $\calI_{\la}\in \calL(l^{r_1}, l^{r_2})$.
    We refer to \cite{GuoFanWuZhao2018SM} for more details about the boundedness of fractional integral operators.
    
    Let $r=q_1/2$, $\theta=1/2$, $\wt{s}=d(1/\wt{p_2}-1/\wt{q_1}+1/\wt{q_2}-1/\wt{p_1})$ with 
    \be
    \wt{q_1}=\fy,\ \frac{1}{\wt{p_1}}=2\big(\frac{1}{p_1}-\frac{1}{q_1}\big),\ \frac{1}{\wt{p_2}}=2\big(\frac{1}{p_2}-\frac{1}{q_1}\big),\ \frac{1}{\wt{q_2}}=2\big(\frac{1}{q_2}-\frac{1}{q_1}\big).
    \ee
    So we have $\wt{q_2}<\wt{p_1}<\wt{p_2}<\wt{q_1}=\i$, 
    and then $\calT\in \calL(l^{(\wt{p_1},\i)}_{v_{\wt{s}}}(\zdd), l^{(\wt{p_2},\wt{q_2})}(\zdd))$.
    Moreover ,we have the relations
    \be
    \frac{1-\th}{\wt{p_1}}+\frac{\th}{r}=\frac{1}{p_1},\ 
    \frac{1-\th}{\wt{p_2}}+\frac{\th}{r}=\frac{1}{p_2},\ 
    \frac{1-\th}{\wt{q_1}}+\frac{\th}{r}=\frac{1}{q_1},\ 
    \frac{1-\th}{\wt{q_2}}+\frac{\th}{r}=\frac{1}{q_2}.
    \ee
    Then the final conclusion $\calT\in \calL(l^{(p_1,q_1)}_{v_s}(\zdd), l^{(p_2,q_2)}(\zdd))$
    follows by the complex interpolation between the endpoint case 
    $\calT\in \calL(l^{(\wt{p_1},\i)}_{v_{\wt{s}}}(\zdd), l^{(\wt{p_2},\wt{q_2})}(\zdd))$
    proved above and the obvious fact
    $\calT\in \calL(l^{(r,r)},l^{(r,r)})$.
    
    Accordingly, from above steps we obtain that 
    \[
    Y=\bigcup_{j=0}^3 Y_j=\bigcup_{j=0}^3 Z_j.
    \]
    Observe that $\bigcup\limits_{j=0}^3 Z_j$ is precisely the set consists of elements satisfying that:
    $1/p_2\leq 1/p_1$ and $s$ satisfies the conditions in Theorem \ref{thm-bd-se}.
    So we complete the proof.
\end{proof}
 
 \begin{proof}[Proof of Theorem \ref{thm-bd-se}]
 	Using Theorem \ref{thm-bd-eq} and the fact $\calD_{2, 1/2}(1\otimes v_s)\sim 1\otimes v_s$, we obtain 
 	\be
 	e^{i\Delta}\in \calL(W^{p_1,q_1}_{1\otimes v_s}(\rd),W^{p_2,q_2}(\rd))
 	\Longleftrightarrow 
 	\calT\in \calL(l^{(p_1,q_1)}_{1\otimes v_s}(\zdd), l^{(p_2,q_2)}(\zdd)),
 	\ee
 	the desired conclusion follows by Proposition \ref{pp-se}.
 	
\end{proof}

\subsection*{Acknowledgements}
This work was supported by the the National Natural Foundation of China [12371100]
and Natural Science Foundation of Fujian Province
[2021J011192, 2022J011241].

\bibliographystyle{abbrv}

\begin{thebibliography}{10}
	
	\bibitem{BenyiGroechenigOkoudjouRogers2007JoFA}
	{\'{A}}.~B{\'{e}}nyi, K.~Gr{\"{o}}chenig, K.~A. Okoudjou, and L.~G. Rogers.
	\newblock {Unimodular Fourier multipliers for modulation spaces}.
	\newblock {\em Journal of Functional Analysis}, 246(2):366--384, may 2007.
	
	\bibitem{BenyiOkoudjou2009BotLMS}
	A.~B\'{e}nyi and K.~Okoudjou.
	\newblock Local well-posedness of nonlinear dispersive equations on modulation
	spaces.
	\newblock {\em Bulletin of the London Mathematical Society}, 41:549--558, 05
	2009.
	
	\bibitem{BhimaniHaque2023JoFA}
	D.~G. Bhimani and S.~Haque.
	\newblock Strong ill-posedness for fractional {H}artree and cubic {NLS}
	equations.
	\newblock {\em Journal of Functional Analysis}, 285(11):110157, 2023.
	
	\bibitem{1979Gewichtsfunktionen}
	H.~G. Feichtinger.
	\newblock Gewichtsfunktionen auf lokalkompakten gruppen.
	\newblock {\em {\"O}sterreich. Akad. Wiss. Math.-Natur. Kl. Sitzungsber. II},
	188(8):451--471, 1979.
	
	\bibitem{Feichtinger1983TRUoV}
	H.~G. Feichtinger.
	\newblock {Modulation Spaces on Locally Compact Abelian Groups}.
	\newblock {\em Technical Report, University of Vienna}, 1983.
	
	\bibitem{GalperinSamarah2004ACHA}
	Y.~V. Galperin and S.~Samarah.
	\newblock {Time-frequency analysis on modulation spaces $M^{p,q}_m,$ $0<p,q\leq
		\infty$}.
	\newblock {\em Applied and Computational Harmonic Analysis}, 16(1):1--18, 2004.
	
	\bibitem{Groechenig2001}
	K.~Gr{\"o}chenig.
	\newblock {\em Foundations of time-frequency analysis}.
	\newblock Appl. Numer. Harmon. Anal. Boston, MA: Birkh{\"a}user, 2001.
	
	\bibitem{GuoChenFanZhao2019MMJ}
	W.~Guo, J.~Chen, D.~Fan, and G.~Zhao.
	\newblock Characterizations of some properties on weighted modulation and
	{W}iener amalgam spaces.
	\newblock {\em Michigan Math. J.}, 68(3):451--482, 2019.
	
	\bibitem{GuoFanWuZhao2018SM}
	W.~Guo, D.~Fan, H.~Wu, and G.~Zhao.
	\newblock {Sharp weighted convolution inequalities and some applications}.
	\newblock {\em Studia Mathematica}, 241(3):201--239, 2018.
	
	\bibitem{GuoZhao2020JoFA}
	W.~Guo and G.~Zhao.
	\newblock Sharp estimates of unimodular {F}ourier multipliers on {W}iener
	amalgam spaces.
	\newblock {\em Journal of Functional Analysis}, 278(6):108405, 2020.
	
	\bibitem{Heil2003}
	C.~Heil.
	\newblock An introduction to weighted {W}iener amalgams. {I}n {M}. {K}rishna,
	{R}. {R}adha, and {S}. {T}hangavelu,editors.
	\newblock {\em Wavelets and their Applications (Chennai, January 2002)}, pages
	183--216., Allied Publishers, NewDelhi, 2003.
	
	\bibitem{KatoKobayashiIto2012TMJSS}
	K.~Kato, M.~Kobayashi, and S.~Ito.
	\newblock Representation of {S}chr{\"o}dinger operator of a free particle via
	short-time {F}ourier transform and its applications.
	\newblock {\em Tohoku Mathematical Journal, Second Series}, 64(2):223--231,
	2012.
	
	\bibitem{KatoTomita2018MN}
	T.~Kato and N.~Tomita.
	\newblock A remark on the {S}chr{\"o}dinger propagator on {W}iener amalgam
	spaces.
	\newblock {\em Mathematische Nachrichten}, 292:350 -- 357, 2018.
	
	\bibitem{Schippa2022JoFA}
	R.~Schippa.
	\newblock On smoothing estimates in modulation spaces and the nonlinear
	{S}chr{\"o}dinger equation with slowly decaying initial data.
	\newblock {\em Journal of Functional Analysis}, 282(5):109352, 2022.
	
	\bibitem{Triebel1983ZfAuiA}
	H.~Triebel.
	\newblock {Modulation Spaces on the Euclidean n-Space}.
	\newblock {\em Zeitschrift f{\"{u}}r Analysis und ihre Anwendungen},
	2(5):443--457, 1983.
	
	\bibitem{Walnut1992JMAA}
	D.~F. Walnut.
	\newblock Continuity properties of the {G}abor frame operator.
	\newblock {\em J. Math. Anal. Appl.}, 165(2):479--504, 1992.
	
	\bibitem{WangHudzik2007JDE}
	B.~Wang and H.~Hudzik.
	\newblock {The global Cauchy problem for the NLS and NLKG with small rough
		data}.
	\newblock {\em Journal of Differential Equations}, 232(1):36--73, jan 2007.
	
	\bibitem{WangZhaoGuo2006JoFA}
	B.~Wang, L.~Zhao, and B.~Guo.
	\newblock Isometric decomposition operators, function spaces
	$e_{p,q}^{\lambda}$ and applications to nonlinear evolution equations.
	\newblock {\em Journal of Functional Analysis}, 233(1):1--39, 2006.
	
	\bibitem{ZhaoChenFanGuo2016NLAT}
	G.~Zhao, J.~Chen, D.~Fan, and W.~Guo.
	\newblock {Sharp estimates of unimodular multipliers on frequency decomposition
		spaces}.
	\newblock {\em Nonlinear Analysis}, 142:26--47, sep 2016.
	
\end{thebibliography}

\end{document}